\documentclass[11pt]{article}

\usepackage[T1]{fontenc}
\usepackage[english]{babel}
\usepackage[babel]{microtype}
\usepackage[a4paper, margin=1.1in]{geometry}

\usepackage{amsmath,amssymb,amsthm,amsrefs}

\usepackage{mathrsfs}

\usepackage{mathtools}
\usepackage{commath}

\usepackage{bbm}

\usepackage{cases}
\usepackage{enumitem}
\setlist[enumerate,1]{label=(\roman*)}

\usepackage{centernot}
\usepackage{booktabs}
\usepackage{multirow}
\usepackage{comment}
\usepackage{float}
\usepackage{scalerel}

\usepackage[pdftex, pdfborderstyle={/S/U/W 0}]{hyperref}
\hypersetup{
    colorlinks=true,
    linkcolor=magenta,
    citecolor=cyan,
}

\usepackage{cleveref}

\usepackage{etoolbox}

\usepackage{subfiles}
\usepackage{pdfpages}

\numberwithin{equation}{section}

\theoremstyle{plain}

\newtheorem{theorem}    {Theorem}[section]
\newtheorem{lemma}      [theorem] {Lemma}
\newtheorem{conjecture} [theorem] {Conjecture}

\newtheorem{corollary}  [theorem] {Corollary}

\theoremstyle{definition}
\newtheorem{definition} [theorem] {Definition}

\newtheorem{claim}      [theorem] {Claim}
\newtheorem{question}   [theorem] {Question}

\newtheorem{remark}    [theorem]{Remark}

\newcommand{\defined}{\mathrel{\coloneqq}}

\renewcommand{\le}{\leqslant}
\renewcommand{\leq}{\leqslant}
\renewcommand{\ge}{\geqslant}
\renewcommand{\geq}{\geqslant}

\let\oldexists\exists
\let\exists\relax
\DeclareMathOperator{\exists}{\oldexists}
\let\oldforall\forall
\let\forall\relax
\DeclareMathOperator{\forall}{\oldforall}

\newcommand{\st}{\mathbin{\colon}}
\undef{\set}
\DeclarePairedDelimiter{\set}{\lbrace}{\rbrace}

\DeclarePairedDelimiter{\card}{\lvert}{\rvert}

\DeclareMathOperator{\ind}{\mathbf{1}}

\newcommand{\from}{\colon}

\undef{\mod}
\newcommand{\mod}[1]{\ (\mathrm{mod}\ #1)}

\undef{\abs}
\DeclarePairedDelimiterX{\abs}[1]
  {\lvert}{\rvert}{\ifblank{#1}{\,\cdot\,}{#1}}
\undef{\norm}
\DeclarePairedDelimiterX{\norm}[1]
  {\lVert}{\rVert}{\ifblank{#1}{\,\cdot\,}{#1}}
\DeclarePairedDelimiterX{\inner}[2]
  {\langle}{\rangle}{\ifblank{#1}{\,\cdot\,}{#1},\ifblank{#2}{\,\cdot\,}{#2}}

\DeclareMathOperator{\Bin}{Bin}

\DeclareMathDelimiter{\given}
  {\mathbin}{symbols}{"6A}{largesymbols}{"0C}
\DeclareMathOperator{\Prob}{\mathbb{P}}
\DeclarePairedDelimiterXPP{\prob}[1]
  {\Prob}{\lparen}{\rparen}{}
  {\renewcommand{\given}{\nonscript\;\delimsize\vert\nonscript\;\mathopen{}}
  \ifblank{#1}{\,\cdot\,}{#1}}
\DeclareMathOperator{\Expec}{\mathbb{E}}
\DeclarePairedDelimiterXPP{\expec}[1]
  {\Expec}{\lparen}{\rparen}{}
  {\renewcommand{\given}{\nonscript\;\delimsize\vert\nonscript\;\mathopen{}}
  \ifblank{#1}{\,\cdot\,}{#1}}
\DeclareMathOperator{\Var}{Var}
\DeclarePairedDelimiterXPP{\var}[1]
  {\Var}{\lparen}{\rparen}{}
  {\renewcommand{\given}{\nonscript\;\delimsize\vert\nonscript\;\mathopen{}}
  \ifblank{#1}{\,\cdot\,}{#1}}
\DeclareMathOperator{\Cov}{Cov}
\DeclarePairedDelimiterXPP{\cov}[2]
  {\Cov}{\lparen}{\rparen}{}{#1,#2}

\newcommand{\EE}{\mathbb{E}}

\newcommand{\NN}{\mathbb{N}}

\newcommand{\RR}{\mathbb{R}}

\newcommand{\cB}{\mathcal{B}}

\newcommand{\cM}{\mathcal{M}}

\newcommand{\cP}{\mathcal{P}}

\newcommand{\cW}{\mathcal{W}}

\title{Counting Hamiltonian Cycles in Dirac Hypergraphs}
\author{Asaf Ferber \thanks{Department of Mathematics, University of California, Irvine.
		Email: \href{mailto:asaff@uci.edu} {\nolinkurl{asaff@uci.edu}}.
		Research supported in part by NSF Awards DMS-1954395 and DMS-1953799.}\and Liam Hardiman \thanks{Department of Mathematics, University of California, Irvine.
		Email: \href{mailto:lhardima@uci.edu} {\nolinkurl{lhardima@uci.edu}}.}\and Adva Mond \thanks{Department of Pure Mathematics and Mathematical Statistics (DPMMS), University of Cambridge, Wilberforce Road, Cambridge, CB3 0WA, United Kingdom. Email: \href{mailto:am2759@cam.ac.uk}{\nolinkurl{am2759@cam.ac.uk}}}}

\date{\today}

\begin{document}
	
	\maketitle
	
	\begin{abstract}
		
		For $0\leq \ell <k$, a Hamiltonian $\ell$-cycle in a $k$-uniform hypergraph $H$ is a cyclic ordering of the vertices of $H$ in which the edges are segments of length $k$ and every two consecutive edges overlap in exactly $\ell$ vertices.
		We show that for all $0\le \ell<k-1$, every $k$-graph with minimum co-degree $\delta n$ with $\delta>1/2$ has (asymptotically and up to a subexponential factor) at least as many Hamiltonian $\ell$-cycles as in a typical random $k$-graph with edge-probability $\delta$.
		This significantly improves a recent result of Glock, Gould, Joos, K\"uhn, and Osthus, and verifies a conjecture of Ferber, Krivelevich and Sudakov for all values $0\leq \ell<k-1$.
		
	\end{abstract}

	\section{Introduction}
	A classical theorem of Dirac \cite{Dirac} states that any graph on $n\geq 3$ vertices with minimum degree at least $n/2$ is Hamiltonian. We call graphs that meet this minimum degree requirement \emph{Dirac graphs}.
	The complete bipartite graph $K_{n,n+1}$ is an extremal example for the tightness of this minimum degree condition.
	Moreover, since adding one edge to it already creates many Hamiltonian cycles, this suggests that Dirac graphs might contain not only one, but many Hamiltonian cycles.
	This leads to the question of \emph{how many} Hamiltonian cycles are contained in a Dirac graph.
	In a seminal paper by S\'ark\"ozy, Selkow and Szemer\'{e}di~\cite{SSS} it was proven that $n$-vertex Dirac graphs contain at least $c^n n!$ many distinct Hamiltonian cycles for some small positive constant $c$.
	As this is clearly the correct order of magnitude, one could further ask for the correct value of the constant $c$.
	
	To come up with an intelligent guess for the value of $c$, consider a binomial random graph $G_{n,p}$ on $n$ vertices wherein each edge appears independently with probability $p$.
	It is easy to show (for example by Chernoff bounds) that with high probability (i.e., with probability tending to 1 as $n\to \infty$) its minimum degree is $(1-o(1))np$, and that the expected number of Hamiltonian cycles is
	\begin{equation}
	\label{eq:expected}
	\frac{1}{2}(n-1)!p^n=(1-o(1))^n n!p^n.
	\end{equation}
	(It is not so easy to show concentration though! See Janson~\cite{JansonConcentration}.) This hints that we might take $c\approx\delta(G)/n$, where $\delta(G)$ denotes the the minimum degree in $G$. 
	Indeed, Cuckler and Kahn \cite{CucklerKahnCycles} impressively showed that $c\approx\delta(G)/n$ is the correct constant, thereby closing the case completely.
	
	It is natural to extend Cuckler and Kahn's result to the hypergraph setting. First, let us introduce a notion of cycle in hypergraphs.
	For positive integers $0\leq \ell < k$, we define a $(k, \ell)$-cycle to be a $k$-uniform hypergraph (or a ``$k$-graph'' for short) whose vertices may be ordered cyclically such that its edges are segments of length $k$ and every two consecutive edges overlap in exactly $\ell$ vertices.
	A $(k,\ell)$-cycle which contains all the vertices of a given $k$-graph is called a \emph{Hamiltonian $\ell$-cycle}.
	We say that a $k$-graph is \emph{$\ell$-Hamiltonian} if it contains a Hamiltonian $\ell$-cycle.
	When $\ell = k-1$ we often refer to an $\ell$-cycle as a \emph{tight cycle}, and we say that a $k$-graph is \emph{tight Hamiltonian} or contains a \emph{Hamiltonian tight cycle}, accordingly.
	In light of Dirac's theorem, we also consider the more general notion of degrees in hypergraphs.
	We say that the \emph{co-degree} of a $(k-1)$-set $X$ in a $k$-graph $H$ is the number of edges in $H$ that contain $X$.

	There has been much work on analogues of Dirac's theorem in the hypergraph setting. Initial results were due to Katona and Kierstead~\cite{KatonaKierstead}.
	Later it was shown that the necessary minimum degree for a $k$-graph $H$ to be $\ell$-Hamiltonian is $\delta_{k-1}(H) \approx n/2$, for $\ell=k-1$~\cite{RRS08,RRS11}, and more generally for $\ell$ satisfying $(k-\ell) \mid k$~\cite{MR11}.
	For values of $\ell$ satisfying $(k-\ell) \nmid k$, it was proven in \cite{KMO10} that the necessary minimum co-degree for $\ell$-Hamiltonicity is $\delta_{k-1}(H) \approx \frac{n}{\lceil k/(k-\ell) \rceil (k-\ell)}$.
	For more details about (many) other results regarding the minimum co-degree of a hypergraph and $\ell$-Hamiltonicity, we refer the reader to the excellent surveys by R\"{o}dl and Ruci\'{n}ski~\cite{RRsurvey} and by K\"{u}hn and Osthus~\cite{KOsurvey}.
	
	In light of these results, and since we consider Hamiltonian $\ell$-cycles for various values of $\ell$, we say that $H$ is \emph{$\delta$-Dirac} if $\delta_{k-1}(H) \ge \delta n$ for $\delta > 1/2$.
	
	A natural guess for the correct lower bound on the number of Hamiltonian $\ell$-cycles in a $\delta$-Dirac graph is the expected number of Hamiltonian $\ell$-cycles in a random hypergraph with edge density $\delta$. That is, we hope to obtain a lower bound of the form
	\begin{equation}
	\label{eq:expectedell}
	(1-o(1))^n \cdot \Psi_k(n, \ell) \cdot \delta^{\frac{n}{k-\ell}},
	\end{equation}
	where $\Psi_k(n, \ell)$ is the number of Hamiltonian $\ell$-cycles in the complete $k$-graph on $n$ vertices.
	Ferber, Krivelevich and Sudakov \cite{FerberKrivelevichSudakov16} realized this hope in the case where $\ell \le k/2$.
	Quite recently, Glock, Gould, K\"{u}hn and Osthus \cite{GGJKO} showed that a $\delta$-Dirac $k$-graph contains at least $(1-o(1))^n n! c^n$ Hamiltonian $\ell$-cycles, for all values $\ell$ and for some small constant $c>0$.
	Our contribution is that (\ref{eq:expectedell}) is the correct lower bound for all values of $0\leq \ell < k-1$ in any $\delta$-Dirac $k$-graph.
	
	\begin{theorem}
		\label{thm:main}
		Let $\ell,k\in \mathbb{N}$ be such that $0 \le \ell < k-1$, and let $n$ be a sufficiently large integer which is divisible by $k-\ell$.
		Then the number of Hamiltonian $\ell$-cycles in a $\delta$-Dirac $k$-graph $H$ on $n$ vertices is at least
		\[
		(1-o(1))^n \cdot \Psi_k(n, \ell) \cdot \delta^{\frac{n}{k-\ell}}.
		\]
	\end{theorem}

\section{Proof outline}
	For $m\in \NN$ and $X\subset V(H)$ we define an $(\ell,m,X)$-\emph{path-system} to be an ordered collection of $m$ many vertex-disjoint $\ell$-paths that cover $X$, and let $\cP(\ell,m,X)$ be the collection of all $(\ell,m,X)$-path-systems in $H$ (see \Cref{def:ellpathsstm}).
	Our proof is largely based on the following three steps:
	$(i)$ We remove a small subset $W\subset V$ with certain properties. This set will be used to tailor path-systems into Hamiltonian cycles.
	$(ii)$ We show that, for an appropriate choice of $m$,  $\card{\cP(\ell,m,V\setminus W)}$  is at least as large as the number of Hamiltonian $\ell$-cycles we eventually want.
	$(iii)$ We show that each $P\in \cP (\ell,m,V\setminus W)$ can be tailored into a Hamiltonian $\ell$-cycle using the vertices in $W$ in such a way that distinct path-systems correspond to distinct cycles.
	Clearly, it then follows that the number of Hamiltonian $\ell$-cycles in $H$ is at least the size of $\cP(\ell, m, V\setminus W)$, as required.
	
	\Cref{sec:PthToCcle} is dedicated to steps $(i)$ and $(iii)$, which mainly follow from other results (mostly stated in ~\cite{GGJKO}).
	More specifically, in \Cref{lem:ConnSstm} we prove that our $k$-graph contains such a subset $W\subset V$, and in \Cref{lem:PthsToCycl} we show how to tailor a path-system into a Hamiltonian $\ell$-cycle, using the set $W$.
	
	Our main contribution is in step $(ii)$, where the goal is to construct ``many'' path-systems, each of which covers all the vertices in $V' \defined V\setminus W$.
	As mentioned above, we use $W$ to tailor each path-system into exactly one Hamiltonian $\ell$-cycle.
	Therefore, if we let $x \defined \card{W}$, then we clearly cannot have more than $\left(\Theta(n-x) \right)^{n-x}$ many path-systems in $V'$, as this is at most how many Hamiltonian cycles one can have on $n-x$ vertices.
	Since the desired lower bound on the number of Hamiltonian $\ell$-cycles is of order $\left(\Theta(n) \right)^n$, we require $(n-x)^{-x}=(1-o(1))^n$.
	Therefore, we must choose $x=o\left(\frac{n}{\log n} \right)$.
	Moreover, since we need to tailor $m$ paths together using vertices from $W$ we also must have that $m=O(x)$.
	
	For convenience we only count (and construct) path-systems in which all paths are of exactly the same length $s \defined \frac{n-x}{m}$ (so from the above discussion we must have $s=\omega(\log n)$).
	The process of constructing a path-system goes as follows.
	First, we choose an ordered equipartition $V' = V_1\cup \ldots \cup V_s$.
	Second, we choose an ``ordered'' perfect matching $M_1=(e_1,\ldots,e_m)$ in the $k$-partite $k$-graph induced by $V_1\cup \ldots \cup V_k$, and for each $i \in [m]$ we let $X_{e_i}$ be the vertices in $e_i$ that are contained in the last $\ell$ parts $V_{k-\ell+1}, \ldots, V_k$.
	Next, we choose a perfect matching $M_2=(f_1,\ldots,f_m)$ in the $k$-partite $k$-graph induced by $V_{k-\ell+1}\cup \ldots \cup V_{2k-\ell}$ in such a way that for each $i$ we have $f_i\cap e_i=X_{e_i}$, and define $X_{f_i}$, analogously, to be the intersection of $f_i$ with the last $\ell$ parts $V_{2(k-\ell)+1}, \ldots, V_{2k-\ell}$, for each $i \in [m]$.
	We repeat this, choosing a perfect matching $M_3 = (g_1, \ldots, g_m)$ in the $k$-partite $k$-graph induced by $V_{2(k-\ell)+1}, \ldots, V_{3k-2\ell}$ such that $g_i \cap f_i = X_{f_i}$ and define $X_{g_i}$ analogously.
	We continue this way, considering the next $k$ parts of the partition in steps of size $k-\ell$, until we cover $V'$.
	Clearly, the union of all the $M_i$'s is an $(\ell, m, V')$-path-system (the order of the paths is induced by the order on $M_1$).
	
	Our goal is to show that this process yields many path-systems.
	Hence, the main ``building block'' in our counting argument will be finding many perfect matchings in $k$-partite $k$-graphs, where the intersection of each edge with the first $\ell$ parts is determined.
	\Cref{sec:PMs} is dedicated to showing that this is possible when considering a $\delta$-Dirac $k$-partite $k$-graph (where the Dirac property applies only for $(k-1)$-sets with vertices in distinct parts).
	We prove \Cref{lem:kDirac}, and as a consequence we get \Cref{cor:ManyPM}, showing that many perfect matchings can indeed be found in each step of this process.
	
	Having proved our ``building block'' in \Cref{cor:ManyPM}, we describe in detail the process of constructing many path-systems in \Cref{sec:ellpaths}.
	We first prove \Cref{lem:ManyPrtns}, showing that for an appropriate choice of the parameter $m$, most of the ordered equipartitions inherit the Dirac-property of our $k$-graph (as $s$-partite induced $k$-graphs).
	We then prove \Cref{lem:ManyPths}, showing that for each such ``good'' equipartition we can construct many distinct path-systems by concatenating perfect matchings from each step in process described above.
	We conclude the section with \Cref{cor:ManyStms}, where we combine both lemmata to get the ``correct'' number of path-systems in our $k$-graph.
	
	Lastly, in \Cref{sec:proof} we tie everything together, showing that \Cref{lem:ConnSstm}, \Cref{cor:ManyStms}, and \Cref{lem:PthsToCycl} imply our result for an appropriate choice of parameters.

\section{Auxiliary results} 
	
	\subsection{Concentration inequalities}
	We use two probabilistic tools.
	The first one is the known result by Chernoff, bounding the lower and the upper tails of the Binomial distribution (see~\cite{AlonSpencer,JRLrandomgraphs}).
	\begin{lemma}[Chernoff bound]
		\label{Chernoff}
		Let $X\sim \Bin(n, p)$ and let $\EE[X] = \mu$.
		Then
		\begin{itemize}
			\item $\Pr[X < (1-\delta)\mu]<e^{-\delta^2\mu/2}$ for every $\delta>0$;
			\item $\Pr[X>(1+\delta)\mu] < e^{-\delta^2\mu/3}$ for every $0<\delta<3/2$.
		\end{itemize}
	\end{lemma}
	
	\begin{remark}
		The above bounds also hold when $X$ is a hypergeometric random variable.
	\end{remark}
	
	Our second probabilistic tool is an application of a concentration inequality by McDiarmid~\cite{McDiarmidAzuma}, proved originally by Maurey~\cite{Maurey} as one of the first uses of concentration inequality outside of probability theory.
	\begin{theorem}
		\label{thm:McDiarmid}
		Let $S_n$ be the group of permutations over a set of $n$ elements, and let $h \from S_n \to \RR$.
		Assume that for some constant $c$ we have that $|h(\pi) - h(\pi')| \le c$ for any $\pi, \pi' \in S_n$ which are obtained from one another by swapping two elements.
		Then for any $t \ge 0$ we have
		\begin{align*}
		\Pr \left[h(\pi) \le \EE(h(\pi)) - t \right] \le \exp \left(-\frac{t^2}{c^2 n} \right).
		\end{align*}
	\end{theorem}

	\subsection{The number of perfect matchings in $k$-partite $k$-graphs}
	\label{sec:PMs}
	In this section we show that in a ``$\delta$-Dirac $k$-partite'' $k$-graph, one can find ``many'' perfect matchings, even when the intersection of each edge in each matching with the first $\ell$ parts is determined.
	This is our main ``building block'' for constructing many path-systems in the proof of the main result.
	We show this as a corollary of a more general statement, \Cref{lem:kDirac}, which is a consequence of the concentration inequality given in \Cref{thm:McDiarmid}.
	The focus of this section is proving \Cref{lem:kDirac}, which allows us to reduce the problem to the bipartite case in graphs.
	This is a version of an idea from \cite{FH}.
	Then, by using a known result by Cuckler and Kahn~\cite{CucklerKahnCycles}, we deduce \Cref{cor:ManyPM}.
	
	We start with introducing some further definitions and notation.
	Let $H$ be a $k$-partite $k$-graph with parts $V_1, \ldots, V_k$, all of size $m$.
	Let $\pi = (\pi_1, \ldots, \pi_{k-1})$, where $\pi_i \from [m] \to V_i$ is a permutation on the vertices in $V_i$ for each $i \in [k-1]$.
	Let $\cM_{\pi}$ be the collection of $(k-1)$-sets that intersect all $V_1, \ldots, V_{k-1}$, induced by $\pi$.
	More precisely, let
	\begin{align*}
	\cM_{\pi} \defined \set[\big]{\set{\pi_1 (j), \ldots, \pi_{k-1}(j) } \st j \in [m]}.
	\end{align*}
	Define the auxiliary graph $\cB_{\pi}(H)$ to be the bipartite graph with parts $\cM_{\pi}$ and $V_k$, and such that $xv$ is an edge for $x \in \cM_{\pi}$ and $v\in V_k$ if and only if $x \cup v \in E(H)$.
	Taking further $\pi_k \from [m] \to V_k$ to be a permutation on the vertices in $V_k$, we say that the tuple of $k$ permutations $(\pi_1, \ldots, \pi_k)$ \emph{induces} a perfect matching in $H$, if the set of edges
	\begin{align*}
	\set[\big]{\set{\pi_1(j), \ldots, \pi_k(j)} \st j\in [m]}
	\end{align*}
	is a perfect matching in $H$.
	
	When considering an $s$-partite $k$-graph $H$, for some $s\ge k$, it is simpler to use the following variant of the notion of minimum co-degree.
	Assume that $V_1, \ldots, V_s$ are the parts of $H$.
	For $i \in [s]$ define
	\begin{align*}
	U_i \defined \bigcup V_{j_1} \times \cdots \times V_{j_{k-1}},
	\end{align*}
	where the union goes over all $1 \le j_1 < \ldots < j_{k-1} \le s$ such that $j_1, \ldots, j_{k-1} \neq i$.
	Define further
	\begin{align*}
	\delta^*_{k-1}(H) \defined \min \set[\big]{d(X, V_i) \st i \in [s], \; X \in U_i},
	\end{align*}
	where $d(X, U_i)$ is the number of edges in $H$ that contain $X$ (viewed as a $(k-1)$-set) and intersect $U_i$ non-trivially.
	That is, $\delta_{k-1}^*(H)$ is the minimum number of edges incident to a $(k-1)$-set that intersects exactly $k-1$ parts.
	
	We can now state the main lemma of this section.
	\begin{lemma}
		\label{lem:kDirac}
		For every $\varepsilon > 0$ there exists $m_0 \in \NN$ such that the following holds for any integer $m \ge m_0$.
		Let $H$ be a $k$-partite $k$-graph with parts $V_1\cup \ldots \cup V_k$, all of size $m$.
		Suppose that $\delta^*_{k-1}(H) \ge \delta m$ for some $\varepsilon < \delta \le 1$.
		For every $i \in [k-1]$ let $\pi_i \from [m] \to V_i$ be a permutation on the vertices of $V_i$, such that $\pi_1, \ldots, \pi_{k-2}$ are fixed and $\pi_{k-1}$ is chosen uniformly at random.
		Denote $\pi \defined (\pi_1, \ldots, \pi_{k-1})$.
		Then with high probability the bipartite graph $\cB_{\pi}(H)$ has minimum degree at least $(\delta - \varepsilon) m$.
	\end{lemma}
	
	\begin{proof}
		Recall that the graph $\cB_{\pi} \defined \cB_{\pi}(H)$ has parts $\cM_{\pi}$ and $V_k$, and note that for every $x \in \cM_{\pi}$ we have $d_{\cB_{\pi}}(x) \ge \delta^*_{k-1}(H) \ge \delta m$.
		So it is left to show that the statement holds for vertices in $V_k$.
		Let $v \in V_k$ and consider $d_{\cB_{\pi}}(v)$.
		Since $\pi_{k-1}$ is chosen uniformly at random, for each $j \in [m]$ we have
		\begin{align*}
		\EE_{\pi_{k-1}}\left[\ind_{\left\{\{\pi_1(j) \ldots, \pi_{k-1}(j), v \} \in E(H) \right\}} \right] &= \Pr\left[\{\pi_1(j), \ldots, \pi_{k-1}(j), v \} \in E(H) \right] \\
		&= \frac{d_H \left(\{\pi_1(j), \ldots, \pi_{k-2}(j), v\}, V_{k-1} \right)}{m} \\
		&\ge \delta.
		\end{align*}
		Thus we have
		\begin{align*}
		\EE_{\pi_{k-1}}\left[d_{\cB_{\pi}}(v) \right] = \sum_{j \in [m]}\EE_{\pi_{k-1}} \left[\ind_{\left\{\{\pi_1(j), \ldots, \pi_{k-1}(j), v \} \in E(H) \right\}} \right] \ge \delta m.
		\end{align*}
		
		Now note that swapping any two elements in $\pi_{k-1}$ can change $d_{\cB_{\pi}}(v)$ by at most $2$.
		Thus, by \Cref{thm:McDiarmid} we get that
		\begin{align*}
		\Pr\left[d_{\cB_{\pi}}(v) \le (\delta - \varepsilon) m \right] \le \Pr\left[d_{\cB_{\pi}(v)} \le \EE\left[d_{\cB_{\pi}}(v) \right] - \varepsilon m \right] \le \exp\left(-\frac{\varepsilon^2 m}{4} \right) = o(1).
		\end{align*}
		Hence, with high probability the minimum degree in $\cB_{\pi}(H)$ is at least $(\delta - \varepsilon) m$.
	\end{proof}
	
	In \cite{CucklerKahnCycles} the authors provide a lower bound on the number of perfect matchings in a Dirac graph, given naturally by the lower bound on the number of Hamiltonian cycles in those graphs.
	
	\begin{theorem}[Theorems 1.5 and 3.1 in \cite{CucklerKahnCycles}]
		\label{thm:CKmtch}
		Let $G$ be a bipartite Dirac graph on parts of size $n$, and with minimum degree $d \ge n/2$.
		Then $G$ contains at least
		\begin{align*}
		(1-o(1))^n \cdot n! \cdot \left(\frac{d}{n} \right)^n
		\end{align*}
		many perfect matchings.
	\end{theorem}
	
	Combining \Cref{lem:kDirac} and \Cref{thm:CKmtch} we get the following corollary.
	\begin{corollary}
		\label{cor:ManyPM}
		Let $H$ be a $k$-partite $k$-graph with parts $V_1, \ldots, V_k$, all of size $m$, and suppose that $\delta^*_{k-1}(H) \ge \delta m$, for some $1/2 < \delta \le 1$.
		Let $0 \le r \le k-2$, and in case that $r \ge 1$ let further $\pi_1, \ldots, \pi_{r}$ be such that $\pi_i \from [m] \to V_i$ is a fixed permutation on the vertices in $V_i$, for each $i \in [r]$.
		Then there are at least
		\begin{align*}
		(1-o(1))^m (m!)^{k-r} \delta^m
		\end{align*}
		many tuples of permutations $(\pi_{r +1}, \ldots, \pi_k)$ for which $\pi' = (\pi_1, \ldots, \pi_k)$ induces a perfect matching in $H$.
	\end{corollary}
	
	\begin{proof}
		Let $0 < \varepsilon \le \delta - 1/2$ and let $m$ be sufficiently large.
		Fix a set of $k-2-r$ permutations, $\pi_{r+1}, \ldots, \pi_{k-2}$ in case $r \le k-3$, and an empty set for $r=k-2$.
		There are $(m!)^{k-2-r}$ possibilities for choosing this set of permutations.
		By \Cref{lem:kDirac} we know that there are at least $(1-o(1))^m m!$ permutations $\pi_{k-1} \from [m] \to V_{k-1}$ for which, if $\pi = (\pi_1, \ldots, \pi_{k-1})$, then the bipartite graph $\cB_{\pi}(H)$ has minimum degree at least $(\delta - \varepsilon)m \ge m/2$.
		Consider one such $\pi_{k-1}$.
		By \Cref{thm:CKmtch} we get that $\cB_{\pi}(H)$ contains at least $(1-o(1))^m m! \delta^m$ perfect matchings, each can be encoded by a certain permutation $\pi_k \from [m] \to V_k$.
		Moreover, each such perfect matching gives a perfect matching in $H$.
		In total we get that there are at least $(1-o(1))^m (m!)^{k-r} \delta^m$ many tuples $\pi' = (\pi_1, \ldots, \pi_k)$ which induce a perfect matching (not necessarily uniquely) in $H$.
	\end{proof}
	
	\subsection{Constructing many $(\ell, m, V')$-path-systems}
	\label{sec:ellpaths}
	This section is the heart of the argument.
	We show that a subset containing most of the vertices in $H$ can by covered by the ``correct'' number of path-systems, that is, the number of Hamiltonian $\ell$-cycles we aim to find in $H$.
	
	We start with the precise notion of a path-system.
	\begin{definition}
		\label{def:ellpathsstm}
		Let $F$ be a $k$-graph and let $X \subset F$ be a subset of vertices.
		We say that an ordered collection $\cP = (P_1, \ldots, P_m)$ is an \emph{$(\ell, m, X)$-path-system}, if
		\begin{itemize}
			\item $P_i$ is an $\ell$-path in $F$ on $\card{X}/m$ vertices, for each $i \in [m]$,
			\item $\set{P_i}_{i \in [m]}$ are pairwise vertex-disjoint, and their union covers all the vertices in $X$.
		\end{itemize}
	\end{definition}
	Note that two different $(\ell, m, X)$-path-systems can consist of the same family of $\ell$-paths but with different orderings.
	We distinguish between two such families, since they will eventually form two different Hamiltonian $\ell$-cycles when we tailor the paths to one another.
	
	For a partition $\Pi$ of the vertices of a $k$-graph $H$ into $s$ parts, we denote by $H[\Pi]$ the $s$-partite $k$-graph spanned by edges going between parts of $\Pi$.
	We use the following two lemmata to show that a subset of most of the vertices in a $\delta$-Dirac $k$-graph can be covered by many path-systems.
	\begin{lemma}
		\label{lem:ManyPrtns}
		Let $H = (V,E)$ be a $k$-graph on $n$ vertices with minimum co-degree $\delta_{k-1}(H) \ge \delta n$, for some $0 < \delta \le 1$, and let $V' \subset V$ be a fixed subset of vertices of size $n' = n-o(n)$.
		Let $\Pi = (V_1, \ldots, V_{n'/m})$ be an equipartition of $V'$ chosen uniformly at random, into parts of size $m \defined m(n) = \omega(\log n)$.
		Then with high probability $H[\Pi]$ satisfies
		\begin{align}
		\label{eq:GoodPrtn}
		\delta^*_{k-1} (H[\Pi]) \ge (\delta - o(1))m.
		\end{align}
	\end{lemma}
	
	\begin{proof}
		Let $0 < \varepsilon < \delta/2$ and let $n$ be sufficiently large.
		Since $V'$ contains all but $o(n)$ many vertices in $H$, and since $\delta_{k-1}(H) \ge \delta n$, we get
		\begin{align*}
		\delta_{k-1}\left(H[V'] \right) \ge \left(\delta - \varepsilon \right)n'.
		\end{align*}
		
		Let $i \in [n'/m]$ and let $X \in V_{j_1} \times \cdots \times V_{j_{k-1}}$ for some $1 \le j_1 < \ldots < j_{k-1} \le n'/m$ with $j_1, \ldots, j_{k-1} \neq i$.
		We have
		\begin{align*}
		\mu \defined \EE\left[d(X, V_i) \right] = |V_i|\frac{d_{H[V']}(X)}{n'} \ge (\delta - \varepsilon)m.
		\end{align*}
		Hence, by \Cref{Chernoff} we get that
		\begin{align*}
		\Pr\left[d(X, V_i) < (\delta - 2\varepsilon)m \right] &\le \Pr\left[d(X, V_i) < (1-\varepsilon)\mu \right] \le \exp \left(-\frac{1}{2}\varepsilon^2(\delta - \varepsilon) \omega(\log n) \right).
		\end{align*}
		Taking a union bound over all $i\in [n'/m]$ and $X \in \binom{V'}{k-1}$ we get
		\begin{align*}
		\Pr \left[\exists X,\text{ and }i: d(X, V_i) < (\delta - 2\varepsilon)m \right] \le n^{k- \omega(1)} = o(1),
		\end{align*}
		and the statement follows.
	\end{proof}
	
	In other words, \Cref{lem:ManyPrtns} shows that almost all partitions of $H[V']$ into $n'/m$ parts inherit the relative minimum co-degree from $H$ as induced $(n'/m)$-partite $k$-graphs.
	In particular, if $H$ is $\delta$-Dirac, then most of these partition inherit this property.
	Furthermore, we show that those induced $k$-graphs given by ``good'' partitions, can be covered by many path-systems.
	
	\begin{lemma}
		\label{lem:ManyPths}
		Let $H = (V,E)$ be a $k$-graph on $n$ vertices with minimum co-degree $\delta_{k-1}(H) \ge \delta n$, for some $1/2 < \delta \le 1$, and let $V' \subset V$ be a fixed subset of vertices of size $n' = n-o(n)$.
		Let $\Pi = (V_1, \ldots, V_{n'/m})$ be an equipartition of the vertices of $V'$ into parts of size $m \defined m(n) = \omega(\log n)$, satisfying (\ref{eq:GoodPrtn}).
		Then the $(n'/m)$-partite $k$-graph $H[\Pi]$ can be covered by at least
		\begin{align*}
		(1-o(1))^{n'}\cdot \left(m! \right)^{\frac{n'}{m}} \cdot \delta^{\frac{n'}{k-\ell}}
		\end{align*}
		many distinct $(\ell, m, V')$-path-systems.
	\end{lemma}
	
	\begin{proof}
		For the first step we consider the $k$-partite subhypergraph of $H$ induced by the first $k$ parts of $\Pi$, that is $H_1 \defined H[V_1, \ldots, V_k]$.
		By \Cref{cor:ManyPM} with $r=0$ we get that there are at least $(1-o(1))^m \left(m! \right)^k \delta^m$ many $\pi^1 = (\pi^1_1, \ldots, \pi^1_k)$ which give a perfect matching in $H_1$.
		Fix one such $\pi^1 = (\pi^1_1, \ldots, \pi^1_k)$, and let $M_1$ be the ordered perfect matching induced by $\pi^1$, ordered according to $\pi^1_1$.
		So we have $M_1 = (e^1_1, \ldots, e^1_m)$,
		where $e^1_j = \set{\pi^1_1(j), \ldots, \pi^1_k(j)}$.
		Moreover, for each $j \in [m]$, let $X^1_j$ be the intersection of $e^1_j$ with the last $\ell$ parts, that is,
		\begin{align*}
		X^1_j \defined e^1_j \cap \left(V_{k-\ell+1}, \ldots, V_k \right) = (\pi^1_{k-\ell+1}(j), \ldots, \pi^1_k(j)).
		\end{align*}
		
		For the second step, consider the $k$ consecutive parts in $\Pi$ beginning with the last $\ell$ parts of $H_1$.
		More precisely, we look at the $k$-partite induced $k$-graph $H_2 \defined H[V_{k-\ell+1}, \ldots, V_{2k-\ell}]$, and we find there a perfect matching $M_2 = \set(e^2_1, \ldots, e^2_m)$ (ordered according to $\pi^2_1$) such that $e^2_j \cap e^1_j = X^1_j$ for every $j \in [m]$.
		We do this as follows.
		We let $(\pi^2_1, \ldots, \pi^2_{\ell}) = (\pi^1_{k-\ell+1}, \pi^1_k)$ be our fixed $\ell$ permutations on $V_{k-\ell+1}, \ldots, V_k$, respectively.
		Then, by \Cref{cor:ManyPM} with $r=\ell$, we get that there are at least
		\begin{align}
		\label{eq:MnyExtns}
		(1-o(1))^m (m!)^{k-\ell} \delta^m 
		\end{align}
		many tuples $(\pi^2_{\ell+1}, \ldots, \pi^2_k)$ for which $\pi^2 = (\pi^2_1, \ldots, \pi^2_k)$ induces a perfect matching in $H_2$.
		Note further that two distinct such tuples induce two distinct perfect matchings in $H_2$, since $\pi^2_1$ is fixed.
		Hence, we get that the number of perfect matching in $H_2$ which agree with $M_1$ on the first $\ell$ parts is at least as given in (\ref{eq:MnyExtns}).
		Let $M_2 = \set{e^2_1, \ldots, e^2_m}$ be such perfect matching, and for every $j \in [m]$, let $X^2_j$ to be the intersection of $e^2_j$ with the last $\ell$ parts of $H_2$, that is
		\begin{align*}
		X^2_j \defined e^2_j \cap (V_{2(k-\ell)+1}, \ldots, V_{2k-\ell}) = (\pi^2_{k-\ell+1}(j), \ldots, \pi^2_{k}(j)).
		\end{align*}
		Note that indeed we have $e^2_j \cap e^1_j = X^1_j$ for every $j \in [m]$.
		
		We then repeat the above, where in each step we consider the next $k$ parts, overlapping the last $\ell$ parts from the preceding step.
		We do this $\frac{n'/m - \ell}{k - \ell}$ many times, until we cover all parts in $\Pi$.
		
		More formally, for $2 \le s \le \frac{n'/m - \ell}{k-\ell}$ we do the following in the $s$-th step.
		Consider the $k$-partite induced $k$-graph $H_s \defined H[V_{(s-1)(k-\ell)+1}, \ldots, V_{sk - (s-1)\ell}]$, and $\ell$ fixed permutations $(\pi^s_1, \ldots, \pi^s_{\ell}) = (\pi^{s-1}_{k-\ell}, \ldots, \pi^{s-1}_k)$ on $V_{(s-1)(k-\ell)+1}, \ldots, V_{(s-1)k - s\ell}$ given by step $s-1$, respectively.
		By \Cref{cor:ManyPM} with $r=\ell$ we get that the number of tuples $(\pi^s_{\ell+1}, \ldots, \pi^s_k)$ for which $\pi^s = (\pi^s_1, \ldots, \pi^s_k)$ induces a perfect matching in $H_s$ is at least as given in (\ref{eq:MnyExtns}).
		Again, any two distinct such tuples induce two distinct perfect matchings in $H_s$, as we have fixed $\pi^s_1$. We get that the number of perfect matchings $M_s$ in $H_s$ which agree with $M_{s-1}$ on the first $\ell$ parts is as at least as in (\ref{eq:MnyExtns}).
		Let $M_s = (e^s_1, \ldots, e^s_m)$ be one such perfect matching (ordered according to $\pi^s_1$).
		For every $j \in [m]$ let
		\begin{align*}
		X^s_j \defined e^s_j \cap (V_{s(k-\ell)+1}, \ldots, V_{sk - (s-1)\ell}) = (\pi^s_{k-\ell+1}, \ldots, \pi^s_k),
		\end{align*}
		and note that we have $e^s_j \cap e^{s-1}_j = X^{s-1}_j$.
		
		The statement then follows by multiplying by the number of options to extend the paths in each step.
		That is, raising (\ref{eq:MnyExtns}) to the power of $\frac{n'/m}{k-\ell}$, we get that there are at least
		\begin{align*}
		(1-o(1))^{n'} \cdot (m!)^{\frac{n'}{m}} \cdot  \delta^{\frac{n'}{k-\ell}}
		\end{align*}
		many $(\ell, m, V')$-path-systems that cover $H[\Pi]$.
	\end{proof}
	
	Combining the counting in \Cref{lem:ManyPrtns} and in \Cref{lem:ManyPths}, we get the following corollary, giving us with the required number of $(\ell, m, V')$-path-systems covering $H[V']$.
	\begin{corollary}
		\label{cor:ManyStms}
		Let $H = (V,E)$ be a $\delta$-Dirac $k$-graph on $n$ vertices.
		Let $V' \subset V$ be a subset of size $n' = n - o\left(\frac{n}{\log n} \right)$.
		Then $H[V']$ can be covered by at least
		\begin{equation}
		\label{eq:ManyStms}
		(1-o(1))^n \cdot \Psi_k(n, \ell) \cdot \delta^{\frac{n}{k-\ell}}
		\end{equation}
		many $(\ell, m, V')$-path-systems.
	\end{corollary}
	
	\begin{proof}
		By \Cref{lem:ManyPrtns} there are at least
		\begin{align*}
		(1-o(1)) \cdot  \frac{n'!}{\left(m! \right)^{n'/m}}
		\end{align*}
		many partitions $\Pi$ which satisfy (\ref{eq:GoodPrtn}).
		By \Cref{lem:ManyPths}, each such partition can be covered by at least $(m!)^{\frac{n'}{m}} \delta^{\frac{n'}{k-\ell}} (1-o(1))^{n'}$ many distinct $(\ell, m, V')$-path-systems.
		However, for some values of $\ell$, the same path-system can be obtained from many different orderings.
		We denote by $c_k(\ell)$ the number of ways to reorder the first $k-\ell$ vertices in each edge such that the path-system is not changed, so we get $(c_k(\ell))^{\frac{n'}{k-\ell}}$ many such reorderings.
		Considering this double counting, and recalling that $n' = n - o\left(\frac{n}{\log n} \right)$, we get that in total $H[V']$ can be covered by at least
		\begin{align*}
		(1-o(1))^{n'} \cdot n'! \cdot \left(\frac{\delta}{c_k(\ell)} \right)^{\frac{n'}{k-\ell}} = (1-o(1))^n \cdot n! \cdot \left(\frac{\delta}{c_k(\ell)} \right)^{\frac{n}{k-\ell}}
		\end{align*}
		many $(\ell, m, V')$-path-systems.
		In order that this is precisely the expression in (\ref{eq:ManyStms}), it is sufficient to prove the following claim.
		\begin{claim}
			$\Psi_k(n, \ell) = (1-o(1))^n \cdot n! \cdot c_k(\ell)^{-\frac{n}{k-\ell}}$.
		\end{claim}
		
		\begin{proof}
			We know that the complete $k$-graph, any cyclical ordering of the vertices yields a Hamiltonian $\ell$-cycle.
			However, the same Hamiltonian $\ell$-cycle can be obtained by two different cyclical orderings of the vertices, for some values of $\ell$.
			Moreover, the quantity $c_k(\ell)$ gives us the correct counting here as well, meaning that a Hamiltonian $\ell$-cycle can be obtained from $c_k(\ell)^{\frac{n}{k-\ell}}$ many different cyclical orderings of the vertices.
			In addition, there are $\frac{2n}{k-\ell}$ different ways to cyclically reorder all $\frac{n}{k-\ell}$ edges and reverse their direction.
			So we have
			\begin{align*}
			\Psi_k(n, \ell) = (n-1)! \cdot \frac{k-\ell}{2} \cdot \left(\frac{1}{c_k(\ell)} \right)^{\frac{n}{k-\ell}}. 
			\end{align*}
			
			To complete the proof, we also describe the quantity $c_k(\ell)$ explicitly.
			Let $0\le r < k-\ell$ be the residue from dividing $k$ by $k-\ell$.
			That is, $r$ satisfies $k = (k-\ell)\left\lfloor\frac{k}{k-\ell} \right\rfloor + r$.
			Then we have
			\begin{equation}
			\label{eq:cknell}
			c_k(\ell) = r!(k-\ell-r)!.
			\end{equation}
			For example, for $\ell < k/2$ we have $r = \ell$, which gives $c_k(\ell) = \ell!(k-2\ell)!$.
			
			Note that, in particular, we have $\Psi_k(n, \ell) = \left(\Theta(n) \right)^n$.
		\end{proof}
		
		All in all, we get that there are at least
		\begin{align*}
		(1-o(1))^n \cdot \Psi_k(n, \ell) \cdot \delta^{\frac{n}{k-\ell}}
		\end{align*}
		many $(\ell, m, V')$-path-systems covering $H[\Pi]$, as required.
	\end{proof}

	\subsection{Turning $(\ell, m , V')$-path-systems into Hamiltonian $\ell$-cycles}
	\label{sec:PthToCcle}
	In this section we show how to form a Hamiltonian $\ell$-cycle from an $(\ell, m, V')$-path-system.
	We do this in two steps.
	We first show that we can put aside a subset of vertices with certain properties which later allow us to use this subset to connect $\ell$-paths to one another.
	We then show how to use this set of vertices to form a Hamiltonian $\ell$-cycle from $\ell$-paths which cover the remaining vertices.
	
	\subsubsection{Finding a $(\delta - o(1), m, t)$-connecting-set in a Dirac hypergraph}
	We start by defining the properties of a set to put aside which will be used to turn a collection of $\ell$-paths into a Hamiltonian $\ell$-cycle.
	In fact, we put aside a family of subsets, so it will later be easier to describe the process of connecting $\ell$-paths to one another precisely.
	\begin{definition}
		\label{def:ConnSstm}
		Let $F$ be a $k$-graph, $\eta \in (0,1)$, and $m, t \ge 1$ integers.
		We say that an ordered collection $\cW = \left(W_1, \ldots, W_m \right)$ is an \emph{$(\eta, m, t)$-connecting-system} in $F$, if
		\begin{enumerate}
			\item $\set{W_i}_{i \in [m]}$ are pairwise disjoint subsets of vertices in $F$,
			\item $\card{W_i} = t$ for all $i \in [m]$, and 
			\item $\forall i \in [m]$ and $\forall X \in \binom{\, V(F)}{k-1}$ we have
			\begin{align*}
			d_{F}(X, W_i) \ge \eta \card{W_i}.
			\end{align*}
		\end{enumerate}
		We say that a subset $W \subset V(F)$ is an \emph{$(\eta, m, t)$-connecting-set} if it admits an equipartition $W = W_1 \cup \cdots \cup W_m$ such that $\left(W_1, \ldots, W_m \right)$ is an $(\eta, m, t)$-connecting-system.
	\end{definition}
	
	We now show that we can find such a set in $H$ where $\eta$ is asymptotically the minimum co-degree of $H$.
	\begin{lemma}
		\label{lem:ConnSstm}
		Let $H = (V,E)$ be a $k$-graph on $n$ vertices with minimum co-degree $\delta n$ for some $0 < \delta \le 1$.
		Then there exists a $(\delta - o(1), m, t)$-connecting-system $\cW = \left(W_1, \ldots, W_m \right)$ in $H$, for any $t = \omega(\log n)$ and $1 \le m \le n/t$.
	\end{lemma}
	
	\begin{proof}
		Let $t = \omega(\log n)$ be an integer, and let $V = W_1 \cup \cdots \cup W_{n/t}$ be an ordered equipartition of $V$ into subsets of size $t$, chosen uniformly at random.
		In fact, we show that with high probability $\left(W_1, \ldots, W_{n/t} \right)$ is a $(\delta - o(1), n/t, t)$-connecting-system for $H$.
		
		The proof is very similar to that of \Cref{lem:ManyPrtns}.
		For every $X \in \binom{V}{k-1}$ and for every $i$, the co-degree $d(X, W_i)$ is a hypergeometrically distributed random variable with
		\begin{align*}
		\mu \defined \mathbb E \left[d(X, W_i) \right] = \frac{d_H(X)}{n} \card{W_i} \ge \delta t.
		\end{align*}
		Fix $X \in \binom{V}{k-1}$, and $i \in [n/t]$.
		By \Cref{Chernoff} with $0 < \varepsilon < \delta$ and by the above, we have
		\begin{align*}
		\Pr\left[d(X, W_i) < (\delta - \varepsilon)t \right] &\le \Pr\left[d(X, W_i) < (1-\varepsilon)\mu \right] \\
		&\le \exp\left(-\frac{\varepsilon^2 \delta}{2} \cdot \omega(\log n) \right).
		\end{align*}
		A union bound over all possible $X \in \binom{V}{k-1}$ and $i \in [n/t]$ gives
		\begin{align*}
		\Pr \left[\exists X,\text{ and }i: d(X, W_i) < (\delta - \varepsilon)\card{W_i} \right] \le n^{k-\omega(1)} = o(1).
		\end{align*}
		Hence, we may choose an ordered equipartition $V = W_1 \cup \cdots \cup W_{n/t}$ with the property that for each $i \in [n/t]$ and each $X\in \binom{V}{k-1}$ we have $d_H(X, W_i) \ge (\delta - \varepsilon)\card{W_i}$.
		We can then take $\cW \defined \left(W_1, \ldots, W_m \right)$ to be a $(\delta - \varepsilon, m, t)$-connecting-system in $H$, for any $1 \le m \le n/t$.
	\end{proof}

	\subsubsection{Forming a Hamiltonian cycle using a $(\delta-o(1), m, t)$-connecting-system}
	It is left to show how to use a $(\delta - o(1), m, t)$-connecting-system in $H$, for appropriate parameters $m, t$, to form a Hamiltonian $\ell$-cycle from an $(\ell, m, V')$-path-system that covers the rest of the vertices in $H$.
	For this we need a lemma from \cite{GGJKO} about $\ell$-Hamiltonian connectedness.
	\begin{definition}
		\label{def:HamCnctd}
		A $k$-graph $F$ is \emph{$\ell$-Hamiltonian connected} if for every two disjoint ordered subsets of $k-1$ vertices $\overrightarrow{X}, \overrightarrow{Y} \in V(F)^{k-1}$ there exists a Hamiltonian $\ell$-path in $F$ with end-edges $\overrightarrow{X}$ and $\overrightarrow{Y}$.
	\end{definition}
	
	\begin{lemma}[Lemma 3.7 in \cite{GGJKO}]
		\label{lem:HamCnctd}
		For every $\varepsilon > 0$ there exists $n_0$ such that every $k$-graph $F$ on $n \ge n_0$ vertices with minimum co-degree at least $(1/2 + \varepsilon)n$ is $(k-1)$-Hamiltonian connected.
	\end{lemma}
	
	\begin{remark}
		\label{rem:HamCnctd}
		Note that any $(k-1)$-Hamiltonian path yields an $\ell$-Hamiltonian path with the same end-edges, for any $0 \le \ell \le k-1$ satisfying $(k-1) | n$.
		This is done simply by keeping every $(k-\ell)$-th edge, starting from one of the end-edges.
		In particular this means that if a $k$-graph is $(k-1)$-Hamiltonian connected, then it is also $\ell$-Hamiltonian connected for any $0 \le \ell \le k-1$ for which $n$ is divisible by $k-\ell$.
	\end{remark}
	
	We are now ready to show that we can connect the paths using the connecting-system we put aside, in a way that guarantees that no two distinct $(\ell, m, V')$-path-systems form the same Hamiltonian $\ell$-cycle.
	\begin{lemma}
		\label{lem:PthsToCycl}
		Let $H = (V,E)$ be a $k$-graph on $n$ vertices with minimum co-degree at least $\delta n$ for some $1/2 < \delta \le 1$.
		Let $W \subset V$ be a $\left(\delta-o(1), m, t \right)$-connecting-set in $H$, where $t \defined t(n) = \omega(1)$ and $m$ is some integer.
		Denote $V' \defined V \setminus W$, and let $\cP = \left(P_1, \ldots, P_m \right)$ be an $(\ell, m, V')$-path-system covering $H[V']$.
		Then there exists a Hamiltonian $\ell$-cycle $C$ in $H$, containing $\set[\big]{P_i}_{i \in [m]}$ as segments, according to their ordering in $\cP$.
		Moreover, if $\cP_1, \cP_2$ are two distinct $(\ell, m, V')$-path-systems, then the Hamiltonian $\ell$-cycles $C_1, C_2$ obtained from them are distinct as well.
	\end{lemma}
	
	\begin{proof}
		Let $W = W_1 \cup \cdots \cup W_m$ be an equipartition of $W$ such that $\cW = \left(W_1, \ldots, W_m \right)$ is a $\left(\delta - \varepsilon, m, t \right)$-connecting-system in $H$, for some $0 < \varepsilon < \delta/2 - 1/4$.
		We use these parts to tailor the $\ell$-paths in $\cP$ to one another.
		For every $i \in [m]$ let $X_i$ and $Y_i$ be the sets of vertices in the first and the last edges of the path $P_i$, respectively, and let $\overrightarrow{X_i}, \overrightarrow{Y_i}$ be the corresponding edges, ordered according to the path $P_i$.
		By \Cref{def:ConnSstm} (iii), for each $i \in [m]$, the induced subgraph $H_i \defined H[W_i \cup Y_i \cup X_{i+1}]$ (where $i$ is taken modulo $m$) has minimum co-degree at least $(\delta - 2\varepsilon)|H_i| > \frac{1}{2} \card{H_i}$.
		Hence, by \Cref{lem:HamCnctd} and \Cref{rem:HamCnctd} it is $\ell$-Hamiltonian connected.
		Let $Q_i$ be a Hamiltonian $\ell$-path in $H'_i$ with end-edges $\overrightarrow{Y_i}$ and $\overrightarrow{X}_{i+1}$.
		Then we get that $C \defined (P_1, Q_1, \ldots, P_m, Q_m)$ is a Hamiltonian $\ell$-cycle in $H$.
		
		Note that in $C$, any two paths $P_i, P_{i+1}$ are separated by the vertices of $W_i$, where $\cW$ is fixed in the process.
		Hence, given $\cW$, any two distinct $(\ell, m, V')$-path-systems form two distinct Hamiltonian $\ell$-cycles.
	\end{proof}

	\section{Proof of Theorem \ref{thm:main}}
	\label{sec:proof}
	We can now put all the ingredients together to quickly derive the proof of our main theorem.
	\begin{proof}[Proof of Theorem \ref{thm:main}]
		Let $n' \defined n - \log^4 n$.
		By \Cref{lem:ConnSstm} there exists a subset $W \subset V$ which is a $\left(\delta - o(1), m, t \right)$-connecting-set, where $m \defined \frac{n'}{n}\log^2 n$ and $t \defined \frac{n}{n'}\log^2 n$.
		Consider the remaining set of vertices $V' \defined V\setminus W$.
		Since $\card{W} = mt = \log^4 n$ we get that $\card{V'} = n'$.
		By \Cref{cor:ManyStms}, there are at least
		\begin{align*}
		(1-o(1))^n \cdot \Psi_k(n, \ell) \cdot \delta^{\frac{n}{k - \ell}}
		\end{align*}
		many $(\ell, m, V')$-path-systems covering $H[V']$.
		By \Cref{lem:PthsToCycl}, each such $(\ell, m, V')$-path-system can be completed to a Hamiltonian $\ell$-cycle in $H$ using the vertices in $W$, such that no two distinct $(\ell, m, V')$-path-system form the same cycle.
		Hence the statement is proved.
	\end{proof}

	\section{Concluding remarks and open problems}
	We highlight the natural barrier in our approach for extending our main result to tight Hamiltonian cycles.
	The main obstacle is extending \Cref{cor:ManyPM}, in which we use as a ``building block'' in our proof, to the case where $r = k-1$.
	More precisely, when restricting the first $k-1$ parts of a $k$-partite $k$-graph, one cannot use \Cref{thm:McDiarmid} to find many perfect matchings, as there is no ``random layer'' anymore.
	In order to solve the tight case, we believe that a different approach should be taken.
	However, it seems reasonable to believe that the number of tight Hamiltonian cycles one can find in a $\delta$-Dirac $k$-graph is consistent with our result for all other values of $\ell$.
	Hence we state here the following conjecture, which is a slight generalization of Conjecture 7.1 in~\cite{GGJKO}.
	\begin{conjecture}
		A $\delta$-Dirac $k$-graph $H$ on $n$ vertices contains at least $(1-o(1))^n \cdot n! \cdot \delta^n$ many tight Hamiltonian cycles.
	\end{conjecture}
	
	Recall that for those values of $\ell$ for which $(k-\ell) \nmid k$ the necessary minimum co-degree in $k$-graph on $n$ vertices is $\mu^*_{k-1}(\ell, n) \defined \frac{n}{\lceil k/(k-\ell) \rceil (k-\ell)} < \frac{n}{2}$.
	It makes sense to ask for the number of Hamiltonian $\ell$-cycles in $k$-graphs with minimum co-degree larger than this threshold.
	\begin{question}
		Let $\ell, k , n$ be integers satisfying $(k-\ell) \nmid k$.
		Let $H$ be a $k$-graph on $n$ vertices with $\delta_{k-1} \ge \delta n$ for some $\delta > \mu^*_{k-1}(\ell, n)/n$.
		What is the number of Hamiltonian $\ell$-cycles in $H$?
	\end{question}
	
	One can also study Dirac-type problems in hypergraphs with respect to other notions of degrees:
	For a $k$-graph $H = (V,E)$ and a subset of $d$ vertices, $X \in \binom{V}{d}$, for some $1 \le d \le k-1$, we define the \emph{$d$-degree} of $X$ to be the number of edges in $E$ containing $X$.
	For example, one can ask what is the minimum co-degree condition which enforces a perfect matching in $k$-graph.
	For integers $n, k, d$ satisfying $1 \le d \le k-1$ we let $m_d(k,n)$ be the smallest integer $m$ such that any $k$-graph $H$ on $n$ vertices with $\delta_d(H) \ge m$ contains a perfect matching.
	$m_d(n,k)$ is unknown for most values of of $d$.
	Define further
	\[
	\mu_{d}\left(k\right)\defined \lim_{n\rightarrow \infty}m_{d}\left(k,n\right)/\binom{n-d}{k-d}
	\]
	to be the parameter that encodes the asymptotic behaviour of $m_{d}(k,n)$.
	Although true, it is not obvious that the limit exists, as was proved by Ferber and Kwan~\cite{FKw}.
	However, $\mu_d(k)$ is unknown for most values of $d,k$ (for example, even the case $d=1$ and $k=6$ is open).
	
	Given a $k$-graph $H$ on $n$ vertices satisfying the minimum $d$-degree condition for perfect matchings, one can ask how many of them can be found in $H$.
	As for $d=k-1$ the answer is the same as the expected number of perfect matchings in a random graph with the same edge-density, it is not clear if the same phenomenon occurs also in cases where $d < k-1$.
	Indeed, as was pointed out to the first author by Lisa Sauermann, this is not the case already for minimum $1$-degree in $3$-graphs, as can be shown in the following construction.
	Consider the bipartite $3$-graph $H$ on parts $X\cup Y$, where $\card{X} = n/3-1$ and $\card{Y} = 2n/3+1$, and all possible edges which intersect $X$ in at least one vertex.
	We note that $\delta_1(H) < \frac{5}{9} \binom{n}{2}$ which was proved in~\cite{HPS09} to be the correct threshold for minimum $1$-degree in $3$-graph.
	Clearly $H$ does not contain a perfect matching, since every set of disjoint edges has size at most $n/3-1$.
	Now, for $\varepsilon>0$ let $H_{\varepsilon}$ be the bipartite $3$-graph defined similarly to $H$, but with parts of size $\card{X} = (1/3+\varepsilon)n$ and $\card{Y} = (2/3-\varepsilon)n$. 
	$H_{\varepsilon}$ satisfies the minimum $1$-degree condition and hence contains a perfect matching.
	A calculation shows that for sufficiently small $\varepsilon>0$ $H_{\varepsilon}$ cannot contain more than $(1+o_{\varepsilon}(1))^n\frac{n!}{(n/3)! (3!)^{n/3}} \cdot \left(\frac{12}{27} \right)^{n/3}$ perfect matchings.
	Note that $H_{\varepsilon}$ has edge-density at least $5/9$, so the above number is smaller than the expected number of perfect matchings in random $3$-graph with the same edge-density.
	
	For $1 \le d < k$, define the function $f(d, k, n)$ to be the number of perfect matchings in a $k$-graph $H$ on $n$ vertices with $\delta_d(H) \ge \delta\binom{n-d}{k-d}$ where $\delta>\mu_d(k)$.
	It would be interesting to understand the behavior of this function for all values $0<d<k$.

	{\bf Acknowledgement.} The first author would like to thank Lisa Sauermann for sharing the observation that was discussed in the concluding remarks section.
	The third author is thankful to Trinity College of the University of Cambridge for the Trinity Internal Graduate Studentship funding.
	\bibliographystyle{abbrv}
	\bibliography{refs}

\begin{bibdiv}
\begin{biblist}

\bib{AlonSpencer}{book}{
      author={Alon, N.},
      author={Spencer, J.H.},
       title={{The Probabilistic Method}},
   publisher={John Wiley \& Sons},
        date={2004},
}

\bib{CucklerKahnCycles}{article}{
      author={Cuckler, B.},
      author={Kahn, J.},
       title={Hamiltonian cycles in {D}irac graphs},
        date={2009},
     journal={Combinatorica},
      volume={29},
       pages={299\ndash 326},
}

\bib{Dirac}{article}{
      author={Dirac, G.A.},
       title={Some theorems on abstract graphs},
        date={1952},
     journal={Proceedings of the London Mathematical Society},
      volume={s3-2},
       pages={69\ndash 81},
}

\bib{FH}{article}{
      author={Ferber, A.},
      author={Hirschfeld, L.},
       title={Co-degrees resilience for perfect matchings in random
  hypergraphs},
        date={2020},
     journal={The Electronic Journal of Combinatorics},
       pages={1\ndash 40},
}

\bib{FerberKrivelevichSudakov16}{article}{
      author={Ferber, A.},
      author={Krivelevich, M.},
      author={Sudakov, B.},
       title={Counting and packing {H}amilton $\ell$-cycles in dense
  hypergraphs},
        date={2016},
     journal={Journal of Combinatorics},
      volume={7},
       pages={135\ndash 157},
}

\bib{FKw}{article}{
      author={Ferber, A.},
      author={Kwan, M.},
       title={Dirac-type theorems in random hypergraphs},
        date={2020},
     journal={arXiv preprint arXiv:2006.04370},
}

\bib{GGJKO}{article}{
      author={Glock, S.},
      author={Gould, S.},
      author={Joos, F.},
      author={K{\"u}hn, D.},
      author={Osthus, D.},
       title={Counting {H}amilton cycles in {D}irac hypergraphs},
        date={2021},
     journal={Combinatorics, Probability and Computing},
      volume={30},
      number={4},
       pages={631\ndash 653},
}

\bib{HPS09}{article}{
      author={Hàn, H.},
      author={Person, Y.},
      author={Schacht, M.},
       title={On perfect matchings in uniform hypergraphs with large minimum
  vertex degree},
        date={2009},
     journal={SIAM Journal on Discrete Mathematics},
      volume={23},
      number={2},
       pages={732\ndash 748},
}

\bib{JansonConcentration}{article}{
      author={Janson, S.},
       title={The numbers of spanning trees, {H}amilton cycles and perfect
  matchings in a random graph},
        date={1994},
     journal={Combinatorics, Probability, and Computing},
      volume={3},
       pages={97\ndash 126},
}

\bib{JRLrandomgraphs}{book}{
      author={Janson, S.},
      author={Ruci\'{n}ski, A.},
      author={{\L}uczak, T.},
       title={Random {G}raphs},
   publisher={John Wiley \& Sons},
        date={2000},
}

\bib{KatonaKierstead}{article}{
      author={Katona, G.Y.},
      author={Kierstead, H.A.},
       title={Hamiltonian chains in hypergraphs},
        date={1999},
     journal={Journal of Graph Theory},
      volume={30},
      number={3},
       pages={205\ndash 212},
}

\bib{KMO10}{article}{
      author={K\"{u}hn, D.},
      author={Mycroft, R.},
      author={Osthus, D.},
       title={Hamilton $\ell$-cycles in uniform hypergraphs},
        date={2010},
     journal={Journal of Combinatorial Theory, Series A},
      volume={117},
      number={7},
       pages={910\ndash 927},
}

\bib{KOsurvey}{article}{
      author={K\"{u}hn, D.},
      author={Osthus, D.},
       title={Hamilton cycles in graphs and hypergraphs: an extremal
  perspective},
        date={2014},
     journal={arXiv preprint arXiv:1402.4268},
}

\bib{MR11}{article}{
      author={Markström, K.},
      author={Ruci\'{n}ski, A.},
       title={Perfect matchings (and {H}amilton cycles) in hypergraphs with
  large degrees},
        date={2011},
     journal={European Journal of Combinatorics},
      volume={32},
      number={5},
       pages={677\ndash 687},
}

\bib{Maurey}{article}{
      author={Maurey, B.},
       title={Construction de suites symétriques},
        date={1979},
     journal={CR Acad. Sci. Paris Sér. AB},
      volume={288},
      number={14},
       pages={A679\ndash A681},
}

\bib{McDiarmidAzuma}{article}{
      author={McDiarmid, C.},
       title={Concentration},
        date={1998},
     journal={Probabilistic methods for algorithmic discrete mathematics},
       pages={195\ndash 258},
}

\bib{RRsurvey}{article}{
      author={R\"{o}dl, V.},
      author={Ruci\'{n}ski, A.},
       title={Dirac-type questions for hypergraphs—a survey (or more problems
  for {E}ndre to solve)},
        date={2010},
     journal={in \textbf{An irregular mind}},
       pages={561\ndash 590},
}

\bib{RRS08}{article}{
      author={R\"{o}dl, V.},
      author={Ruci\'{n}ski, A.},
      author={Szemer\'{e}di, E.},
       title={An approximate dirac-type theorem for $k$-uniform hypergraphs},
        date={2008},
     journal={Combinatorica},
      volume={28},
       pages={229\ndash 260},
}

\bib{RRS11}{article}{
      author={R\"{o}dl, V.},
      author={Ruci\'{n}ski, A.},
      author={Szemer\'{e}di, E.},
       title={Dirac-type conditions for {H}amiltonian paths and cycles in
  3-uniform hypergraphs},
        date={2011},
     journal={Advances in Mathematics},
      volume={227},
      number={3},
       pages={1225\ndash 1299},
}

\bib{SSS}{article}{
      author={S\'{a}rk\"{o}zy, G.N.},
      author={Selkow, S.M.},
      author={Szemer\'{e}di, E.},
       title={On the number of {H}amiltonian cycles in {D}irac graphs},
        date={2003},
     journal={Discrete Mathematics},
      volume={265},
       pages={237\ndash 250},
}

\end{biblist}
\end{bibdiv}

\end{document}